\documentclass{amsart}

\title{On the algebraic $K$-theory of truncated polynomial algebras in several variables}
\author[V. Angeltveit]{Vigleik Angeltveit}
\address{Australian National University, Canberra, Australia}
\email{vigleik.angeltveit@anu.edu.au }
\thanks{The first author was supported by an NSF All-Institutes Postdoctoral Fellowship administered by the Mathematical Sciences Research Institute through its core grant DMS-0441170, by NSF grant DMS-0805917, and by an Australian Research Council Discovery grant.}
\author[T. Gerhardt]{Teena Gerhardt}
\address{Michigan State University, East Lansing, MI 48824}
\email{teena@math.msu.edu}
\thanks{The second author was supported by NSF DMS--1007083 and NSF DMS--1149408}
\author[M.A.Hill]{Michael A.~Hill}
\address{University of Virginia \\ Charlottesville, VA 22904}
\email{mikehill@virginia.edu}
\thanks{The third author was supported by NSF DMS--0906285, DARPA FA9550-07-1-0555, and the Sloan Foundation}
\author[A.Lindenstrauss]{Ayelet Lindenstrauss}
\address{Indiana University \\ Bloomington, IN 47405}
\email{alindens@indiana.edu}

\usepackage{amsxtra}
\usepackage{amsfonts}
\usepackage[latin1]{inputenc}
\usepackage{graphicx}
\usepackage{amsmath,amssymb,latexsym,amsthm,mathrsfs}
\usepackage[all]{xy}
\usepackage{todonotes}

\newtheorem{theorem}{Theorem}[section]
\newtheorem{thm}[theorem]{Theorem}
\newtheorem{lemma}[theorem]{Lemma}

\newtheorem{cor}[theorem]{Corollary}
\newtheorem{defn}[theorem]{Definition}
\newtheorem{proposition}[theorem]{Proposition}
\newtheorem{prop}[theorem]{Proposition}

\newtheorem{remark}[theorem]{Remark}
\newtheorem{example}[theorem]{Example}

\makeatletter
\let\c@equation\c@theorem
\makeatother
\numberwithin{equation}{section}

       \newcommand{\cX}{\mathcal{X}}  

  \newcommand{\bC}{\mathbb{C}}   \newcommand{\bF}{\mathbb{F}}        \newcommand{\bN}{\mathbb{N}}  
\newcommand{\bQ}{\mathbb{Q}}   \newcommand{\bT}{\mathbb{T}}   \newcommand{\bW}{\mathbb{W}}   \newcommand{\bZ}{\mathbb{Z}}

\newcommand{\sma}{\wedge} 
\newcommand{\holim}{\textnormal{holim}}

\newcommand{\hofib}{\textnormal{hofib}}
\newcommand{\hocofib}{\textnormal{hocofib}}

\newcommand{\TR}{\textnormal{TR}}
\newcommand{\TC}{\textnormal{TC}}
\newcommand{\TF}{\textnormal{TF}}
\newcommand{\THH}{\textnormal{THH}}

\newcommand{\vect}{\overrightarrow}

\begin{document}

\begin{abstract}
We consider the algebraic $K$-theory of a truncated polynomial algebra in several commuting variables, $K(k[x_1,\ldots,x_n]/(x_1^{a_1},\ldots,x_n^{a_n}))$. This naturally leads to a new generalization of the big Witt vectors. If $k$ is a perfect field of positive characteristic we describe the $K$-theory computation in terms of a cube of these Witt vectors on $\bN^n$. If the characteristic of $k$ does not divide any of the $a_i$ we compute the $K$-groups explicitly. We also compute the $K$-groups modulo torsion for $k=\bZ$.

To understand this $K$-theory spectrum we use the cyclotomic trace map to topological cyclic homology, and write $\TC(k[x_1,\ldots,x_n]/(x_1^{a_1},\ldots,x_n^{a_n}))$ as the iterated homotopy cofiber of an $n$-cube of spectra, each of which is easier to understand.
\newline

\newline


\end{abstract}

\maketitle

\section{Introduction}
About 15 years ago, Hesselholt and Madsen \cite{HeMa97b} computed the relative algebraic $K$-theory groups of $k[x]/(x^a)$ for a perfect field $k$ of positive characteristic. In this paper we extend their result to a truncated polynomial ring in multiple commuting variables. We study
\[
K(k[x_1,\ldots,x_n]/(x_1^{a_1},\ldots,x_n^{a_n}), (x_1),\ldots,(x_n)),
\]
the appropriate multi-relative version of  $K(k[x_1,\ldots,x_n]/(x_1^{a_1},\ldots,x_n^{a_n}))$. For ease of exposition, let $A=k[x_1,\ldots,x_n]/(x_1^{a_1},\ldots,x_n^{a_n})$, and let $\widetilde{K}(A)$ denote the multi-relative version of $K(A)$.

Hesselholt and Madsen expressed $K_*(k[x]/(x^a), (x))$ in terms of big Witt vectors. Recall that given a truncation set $S \subset \bN$, one can define the Witt ring $\bW_S(k)$. With $S=\{1,2,\ldots,m\}$ this gives the length $m$ big Witt vectors, and they proved that for a perfect field $k$ of positive characteristic 
\[
K_{2q-1}(k[x]/(x^a), (x)) \cong \bW_{aq}(k)/V_a(\bW_q(k))
\]
and
\[
K_{2q}(k[x]/(x^a), (x)) = 0.
\]
Here $V_a$ is the $a$\textsuperscript{th} Verschiebung, one of the structure maps between Witt vectors.

When $k=\bF_p$, it is not much harder to write down the answer explicitly without referring to Witt vectors. However, this forces one to consider the cases $p \mid a$ and $p \nmid a$ separately, and the answer looks somewhat less elegant.

To express the answer in the $n$-variable case, we again take advantage of Witt vectors to organize our calculation. Our computations lead us naturally to an $n$-dimensional version of the big Witt vector construction. We will define these Witt vectors on $\bN^n$ carefully in the next section. As far as we have been able to determine this construction is new; it is not equivalent to the definition given by Dress and Siebeneicher \cite{DrSi88}. 

We say a set $S \subset \bN^n$ is a \emph{truncation set in $\bN^n$} if $(ds_1,\ldots,ds_n) \in S$ implies $(s_1,\ldots,s_n) \in S$, for $d \in \bN$. Given such an $S$, we define the Witt vectors on $\bN^n$, $\bW_S(k)$, in a way that generalizes the construction for $n=1$. We then prove the following theorems computing the homotopy groups of the multi-relative $K$-theory.

Let $S_q(I) \subset \bN^n$ be the truncation set in $\bN^n$ given by
\begin{equation} \label{eq:SqI} 
S_q(I) = \left\{(s_1,\ldots,s_n)\in\mathbb N^{n} \quad \Big| \quad  \sum_{i=1}^n \left\lfloor \frac{s_i-1}{a_i^{\chi_I(i)}} \right\rfloor \leq q-1\right \}.
\end{equation}
Here $\chi_{I}$ is the characteristic function of $I$, evaluating to $1$ if the argument is in $I$ and $0$ otherwise.

\begin{thm}\label{t:main2}
Suppose $k$ is a perfect field of positive characteristic and let
\[
\widehat{E}_{1}^{s,t}=\begin{cases} 
\bigoplus_{|I|=s} \bW_{S_{q}(I)}(k) & t=2q-1 \\
0 & t=2q,
\end{cases}
\]
Then there is a spectral sequence with
\[
E_1^{s,*} \cong \widehat{E}_1^{s,*} \otimes E(x_1,\ldots,x_{n-1}),
\]
where $E(x_1,\ldots,x_{n-1})$ is an exterior algebra over $\bZ$ on $1$-dimensional classes. This spectral sequence collapses at $E_{n}$ and converges to $\widetilde{K}_{t-s}(A)$.

Moreover, if $|k|$ is finite, then
\[
|\widehat{E}_{1}^{s,2q-1}|=\sum_{|I|=s} |k|^{\binom{n+q-1}{n} \prod_{i \in I} a_i}.
\]
\end{thm}

In Hesselholt and Madsen's computation, there was an analogous spectra sequence which automatically collapsed at $E_{2}$, the unique differential given by the Verschiebung 
\[
V_a : \bW_q(k) \to \bW_{aq}(k).
\] 
For a perfect field of positive characteristic, the Verschiebung is injective, and this allows a simple determination of the relative algebraic $K$-groups as a quotient of the Witt vectors by the image of the Verschiebung. In the $n$-dimensional case the Verschiebungs do not capture the whole story. Instead we can say the following.

\begin{thm} \label{t:main3}
Suppose $k$ is a perfect field of characteristic $p>0$ and suppose $p$ does not divide any of the $a_i$. Then the $d_{1}$-differentials in the spectral sequence from Theorem~\ref{t:main2} are split injective. 

In this case, we can write the relative algebraic $K$-groups of $A$ as
\[
\widetilde{K}_{\ast}(A)\cong \widehat{K}_{\ast}(A)\otimes E(x_{1},\dots,x_{n-1}),
\]
where as above $E(x_{1},\dots,x_{n-1})$ is the exterior algebra over $\mathbb Z$ on $1$-dimensional classes, and where $\widehat{K}$ is concentrated in odd degrees with
\[
\widehat{K}_{2q-1}(A) \cong \bW_{S_q(\{1,\ldots,n\})}(k)/\left(\sum_{i=1}^n V^i_{a_i}\big(\bW_{S_q(\{1,\ldots,\hat{i},\ldots,n\})}(k)\big)\right)
\]
(here $\hat{i}$ denotes skipping $i$).

If in addition $k$ is finite then
\[
|\widehat{K}_{2q-1}(A)| = |k|^{\binom{n+q-1}{n} \prod_{i=1}^n (a_i-1)}.
\]
\end{thm}

This completely determines the algebraic $K$-theory groups in this case. We can be even more explicit. Given a truncation set $S$ in $\bN^n$ and $(s_1,\ldots,s_n) \in S$ we can extract a $1$-dimensional ``$p$-typical'' truncation set $p^{\bN_0}(s_1,\ldots,s_n) \cap S$ consisting of those $p$-power multiples of $(s_1,\ldots,s_n)$ which lie in $S$. Then the group $\widehat{K}_{2q-1}(A)$ in Theorem \ref{t:main3} above is isomorphic to
\[
\bigoplus \bW_{p^{\bN_0}(s_1,\ldots,s_n) \cap S_q(\{1,\ldots,n\})}(k;p),
\]
where the sum is over all $s_1,\ldots,s_n \geq 1$ with $p \nmid \gcd(s_i)$ and $a_i \nmid s_i$.

We also compute the ranks of the multi-relative $K$-theory groups of truncated polynomials over $\mathbb{Z}$. We prove the following theorem:
\begin{theorem} \label{PoincareKZ} \label{mainZ}
The Poincar\'e series of the rationalization of
\[
K_*(\bZ[x_1,\ldots,x_n]/(x_1^{a_1} \ldots x_n^{a_n}), (x_1),\ldots,(x_n))
\]
is given by
\[
\sum_{j \geq 1} t^{2j-1} (1+t)^{n-1} \binom{n+j-2}{n-1} \prod_{i=1}^n (a_i-1).
\]
\end{theorem}
This generalizes results of Soul\'{e} \cite{So81} and Staffeldt \cite{St85} in the single variable case. Note that the difference between $\widetilde{K}_*(A)$ for $k=\bZ$ and $k=\bQ$ is torsion, so Theorem \ref{mainZ} also gives the Poincar\'e series of 
\[
K_*(\bQ[x_1,\ldots,x_n]/(x_1^{a_1} \ldots x_n^{a_n}), (x_1),\ldots,(x_n)).
\]

\subsection{Organization}
In \S \ref{s:ndimWitt}, we give a precise definition of the Witt vectors on $\bN^n$. This section might be of independent interest. The later sections provide the proof of Theorems~\ref{t:main2} and \ref{t:main3}, computing these multi-relative $K$-groups via trace methods. In \S \ref{s:cyclotomic}, we define cyclotomic spectra and give some relevant examples. In particular, we describe a cube in cyclotomic spectra, the homotopy pullback of which receives a map from the multi-relative $K$-theory in question. In \S \ref{s:KTheory}, we review the cyclotomic trace map and define the multi-relative $K$-theory spectrum, before we prove the main theorems in \S \ref{s:perfectfields} and \S \ref{s:integers}.

\subsection{Acknowledgments}
The authors would like to thank Bj{\o}rn Dundas and Lars Hesseholt for several helpful conversations regarding this project, and the referee for careful reading of the manuscript and for a greatly tightened Lemma~\ref{lem:Euclid}. The authors would also like to thank the American Institute of Mathematics (AIM). Some of this work was done during a visit to AIM under the SQuaREs program.

\section{Witt vectors} \label{s:ndimWitt}
At the heart of our computation is building an $n$-cube of $S^1$-equivariant spectra, analogous to a cofiber sequence of $S^1$-spectra that Hesselholt and Madsen used in the one variable case. The various vertices of this hypercube are underlain by $n$-fold smash powers of $S^1_+$, and the maps in the diagram arise from the Hesselholt-Madsen maps restricted to one of the factors. This suggests that a useful language will again be Witt vectors, but our Witt vectors will be modeled on truncation sets in $\bN^n$, rather than those modeled on $\bN$.

\subsection{Classical Case}
For the classical construction, recall (e.g.\ from \cite{HeMa97}) that given a truncation set $S \subset \bN = \{1,2,\ldots\}$ and a commutative ring $R$ we can define a commutative ring $\bW_S(R)$. As a set, $\bW_S(R)=R^S$, and we define addition and multiplication in such a way that the ghost map
\[
w: \bW_S(R) \to R^S
\]
that takes a vector $(a_s)_{s \in S}$ to the vector $(w_s)_{s \in S}$ with
\[
w_s = \sum_{dt = s} t a_t^d
\]
is a ring map. One can show that there is exactly one functorial way to do this. 

We will be particularly interested in three families of truncation sets. If our set $S$ is $\{1,2,\ldots,m\}$, we denote $\bW_S(R)$ by $\bW_m(R)$ and call it the big Witt vectors of length $m$. If $S=\{1,p,\ldots,p^{m-1}\}$, we denote $\bW_S(R)$ by $\bW_m(R;p)$ and call it the $p$-typical Witt vectors of length $m$. If $R$ is a $\bZ_{(p)}$-algebra, there is a splitting
\[
\bW_m(R) \cong \prod_{p \nmid s} \bW_{\lfloor \log_p(m/s)+1 \rfloor}(R;p).
\]
Finally, for any $m$, there is a truncation set generated by $m$: 
\begin{equation} \label{eq:anglem} 
\langle m \rangle  = \{d \in \bN,  d \mid m\}
\end{equation}
 and the associated Witt vectors $\bW_{\langle m \rangle}(R)$.

Note that associating Witt vectors to a truncation set is a special case of a more general construction. Consider $\bN$ as a set with an action of the multiplicative monoid $\bN$. A truncation set is then just a subposet of $\bN$ which is closed under divisibility. The $\bN$-action endows this with further structure, though: for any two elements $r$ and $s$ such that $r\leq s$, we have a ``weight'': $s/r\in \bN$. It is possible to build Witt vectors on other weighted posets of this form. For instance, Dress-Siebeneicher's construction of Witt vectors relative to a profinite group  \cite{DrSi88} is built on the poset of finite index subgroups of a profinite group. Our new construction described  below is built on $\bN^n$.

\subsection{Witt vectors on $\bN^n$}
Now we wish to replace $\bN$ by $\bN^n$. Coordinate multiplication gives an action by $\bN$ and hence a weighted poset: $(t_1,\ldots,t_n)$ divides $(s_1,\ldots,s_n)$ if there is some $d \in \bN$ such that $(dt_1,\ldots,dt_n)=(s_1,\ldots,s_n)$, and the number $d$ is the weight. We say that $S \subset \bN^n$ is a truncation set in $\bN^n$ if it is a closed subposet: if $(s_1,\ldots,s_n) \in S$ and $(t_1,\ldots,t_n)$ divides $(s_1,\ldots,s_n)$, then $(t_1,\ldots,t_n) \in S$. From now on we will use vector notation for $n$-tuples of natural numbers. 

Just as in the $1$-dimensional case, we can consider the truncation set generated by a single element:
\begin{equation}\label{eq:anglevects}
\langle\vect{s}\rangle = \big\{ \vect{u}\in\mathbb N^{n}, \vect{u}\mid\vect{s}\big\}.
\end{equation}
This is not the only way to associate $1$-dimensional truncation sets to elements of $\mathbb N^{n}$, however. As a weighted poset, $\bN^n$ splits into a disjoint union of countably infinitely many copies of $\bN$, indexed by those sequences $\vect{s}$ with $\gcd(\vect{s})=1$:
\[
\bN^n=\coprod_{\gcd(\vect{s})=1}\bN\cdot\vect{s}.
\]
Thus our truncation sets in $\bN^n$ are simply disjoint unions of ordinary truncation sets. Given $\vect{s} \in S$, we use the notation $\bN \vect{s} \cap S$ to denote all multiples of $\vect{s}$ in $S$. This can be thought of as a $1$-dimensional truncation set by identifying $\vect{s} \in S$ with $1 \in \bN$.

We can generalize the classical construction of Witt vectors on $\bN$ to Witt vectors on $ \bN^n$, and using the above decomposition of $ \bN^n$ we can understand exactly what our generalization produces:
Given a truncation set $S$ in $\bN^n$, we construct the Witt vectors $\bW_S(R)$. As a set, $\bW_S(R) = R^S$. We define a ghost map
\[ 
w : \bW_S(R) \to R^S
 \]
as the map that takes a vector $(a_{\vect{s}})_{\vect{s} \in S}$ to the vector $(w_{\vect{s}})_{\vect{s} \in S}$ with
\[ 
w_{\vect{s}} = \sum_{d\vect{u}=\vect{s}} \gcd(\vect{u}) (a_{\vect{u}})^d.
\]

\begin{lemma}
There is a unique functorial way to put a ring structure on $\bW_S(R)$ in such a way that the ghost map is a ring map.

Moreover, there is a canonical splitting
\[ 
\bW_S(R) \cong \prod_{\gcd(\vect{s})=1} \bW_{\bN \vect{s} \cap S}(R).
\]
\end{lemma}

\begin{proof}
Since our truncation set splits as a disjoint union of classical truncation sets, this lemma is simply a restatement of two classical facts:
\begin{enumerate}
\item There is a unique functorial ring structure for a classical truncation set, and
\item the functor $S\mapsto R^S$ takes disjoint unions to Cartesian products.
\end{enumerate}
\end{proof}

The classical Witt construction had a great many structure maps linking the Witt vectors for various $S$. Here we have all of the classical ones and more. The real power of this construction is that we have various structure maps which mix the disjoint factors in $\bN^n$.

The restriction map is the easiest to define. Given $S' \subset S$ we get a restriction map
\begin{equation}\label{eq:restrict}
R^S_{S'} : \bW_S(R) \to \bW_{S'}(R)
\end{equation}
by taking the obvious projection $R^S\to R^{S'}$.  Since a truncation set contains all the divisors of any of its elements, this projection commutes with applying the ghost maps, so it is in fact a ring homomorphism.

We have Frobenius maps in each of the $n$ directions. For each $i \in \{1,\ldots,n\}$ and $r \geq 2$, we have a truncation set
\[
S/(1,\ldots,r\ldots,1) = \{(t_1,\ldots,t_n), (t_1,\ldots,rt_i,\ldots,t_n) \in S \}.
\]
We can define a Frobenius map
\[ 
F^i_r : \bW_S(R) \to \bW_{S/(1,\ldots,r,\ldots,1)}(R) 
\]
by requiring that the diagram
\[
 \xymatrix{ \bW_S(R) \ar[r]^-w \ar[d]^{F^i_r} & R^S \ar[d]^{(F^i_r)^w} \\
\bW_{S/(1,\ldots,r,\ldots,1)}(R) \ar[r]^-w & R^{S/(1,\ldots,r\ldots,1)} 
} 
\]
commutes. Here $(F^i_r)^w$ is defined by
\[
(F^i_r)^w(x_{\vect{s}})_{(t_1,\ldots,t_n)} = x_{(t_1,\ldots,rt_i,\ldots,t_n)}.
\]
We note that $F^i_r$ and $F^j_s$ commute if $i \neq j$.

Associated to these Frobenius maps, we also have Verschiebung maps in each of the $n$ directions. We can define a Verschiebung map
\[
V^i_r : \bW_{S/(1,\ldots,r,\ldots,1)}(R) \to \bW_S(R)
\]
by
\[
V^i_r\big((x_{\vect{s}})\big)_{(t_1,\ldots,t_n)} = \begin{cases} x_{(t_1,\ldots,t_i/r,\ldots,t_n)} & \textnormal{if $t_i/r$ is an integer} \\ 0 & \textnormal{otherwise} \end{cases}
\]
Once again we note that $V^i_r$ and $V^j_s$ commute if $i \neq j$.

We need some basic computations with these Witt vectors on $\bN^n$. For the next three propositions, fix a truncation set $S$, let $\vect{s}\in S$ have $\gcd(\vect{s})=1$, and let $S'=\bN\cdot \vect{s}\cap S$. The truncation set $S'$ can be considered as a truncation set in $\bN^n$ or as a classical one, via the identification of  $\bN\cdot \vect{s}$ with $\bN$.  The ghost maps and ring structure we get on $ \bW_{S'}(R)$  are the same in either case, so we will not distinguish between the two. We can use the canonical splitting of $\bN^n$ into its disjoint factors to study the Frobenii and Verschiebungs on individual factors, each of which  can  be thought of as classical Witt vectors:

\begin{proposition}
The map $V^i_r$ splits as a Cartesian product of maps
\[ 
V^i_r : \bW_{S'/(1,\ldots,r,\ldots,1)}(R) \to \bW_{S'}(R),
\]
and similarly for $F^i_r$.
\end{proposition}

This lets us focus attention on the simple factors.  Fix an $i$, and let $d_i=\gcd(s_i,r)$ and let $e_i=r/d_i$.

\begin{proposition}\label{identification}
We have an isomorphism between $S'/(1,\ldots,r,\ldots,1)$ considered as a truncation set in $\bN^n$ and $S'/e_i$ considered as a $1$-dimensional truncation set. 
\end{proposition}

\begin{proof}
In the decomposition of $\bN^n$ into multiples of vectors with coordinates having a greatest common divisor of $1$, $S'/(1,\ldots,r,\ldots,1)$ consists of multiples of $r/d_i \cdot (s_1,\ldots, s_i/r,\ldots, s_n)$, where the multiplier ensures that the $i$'th coordinate an integer.  So when viewed as  a $1$-dimensional truncation set, $S'/(1,\ldots,r,\ldots,1)$ contains exactly those $a$'s for which $a e_i(s_1,\ldots,s_n)\in S'$.
\end{proof}

Since we are expressing our truncation set in terms of classical truncation sets, the classical results tell us the value of the Frobenius and Verschiebung.

\begin{proposition}
If we identify $\bW_{S'/(1,\ldots,r,\ldots,1)}(R)$ with $\bW_{S'/{e_i}} ( R ) $, then $V^i_r$ is given by
\[ 
\bW_{S'/e_i}(R) \xrightarrow{V_{e_i}}\bW_{S'}(R),
\]
where $V_{e_i}$ is the classical Verschiebung.

Similarly, the Frobenius map $F^i_r$ is given by
\[ 
\bW_{S'}(R) \xrightarrow{F_{e_i}} \bW_{S'/e_i}(R).
\]
\end{proposition}

In particular, this means that the composite $F^i_r V^i_r$ is multiplication by $e_i$ on $\bW_{S'/(1,\ldots,r,\ldots,1)}(R)$.

\begin{proof}
Let $\vect{x} = (x_{\vect{t}})_{\vect{t} \in S'/(1,\ldots,r,\ldots,1)}$ be an element of $\bW_{S'/(1,\ldots,r,\ldots,1)}(R)$. Then by definition,
\[
V^i_r(\vect{x})_{a(s_1,\ldots,s_n)}=\vect{x}_{a(s_1,\ldots,s_i/r,\ldots,s_n)}.
\]
Translated under the identification in Proposition \ref{identification} this says that if we view $\vect{x}$ as an element of $\bW_{S'/{e_i}} ( R ) $ (where $S'/e_i$ is considered as a one-dimensional truncation set)  then the $a$'th coordinate of $V^i_r(\vect{x})$ in $\bW_{S'} $ is the $a/e_i$'th coordinate of $\vect{x} \in \bW_{S'/{e_i}} ( R ) $. Thus the map $V^i_r$ agrees with the classical Verschiebung $V_{e_i}$ under the identifications from Proposition \ref{identification}.

In the case of the Frobenius, the ghost maps are the same as in the classical one-dimensional case, and the ghost version of  $F^i_r$ is just the classical ghost version of $F_{e_i}$. Thus $F^i_r$ itself must be equal to the classical $F_{e_i}$.
\end{proof}

\section{cyclotomic spectra} \label{s:cyclotomic}
Certain $S^1$-spectra, called \emph{cyclotomic spectra}, are particularly important to computations of algebraic $K$-theory. We first recall the definition of a cyclotomic spectrum \cite{HeMa97}. Let $G$ be a compact Lie group. Given a genuine $G$-spectrum $X$ indexed on a complete $G$-universe $\mathcal{U}$, and a normal subgroup $H \vartriangleleft G$, there are two notions of $H$-fixed points \cite{LMS}. Both notions of fixed points yield a $G/H$-(pre)spectrum. The first, denoted $X^H$, has $V$\textsuperscript{th} space
\[
X^H(V) = X(V)^H
\] for each $V \subset \mathcal{U}^H$. The second notion, that of geometric fixed points, is defined as follows. Let $\mathcal{F}_H$ denote the family of subgroups of $G$ not containing $H$. Let $E\mathcal{F}_H$ denote the universal space of this family, and let $\widetilde{E}\mathcal{F}_H$ denote the cofiber of the map $E\mathcal{F}_{H+} \to S^0$ given by projection onto the non-basepoint. Then the geometric fixed point spectrum $X^{gH}$ is defined as
\[
X^{gH} = (X\sma \widetilde{E}\mathcal{F}_H)^H.
\]  
The fixed points and geometric fixed points are related as follows. Starting with the cofiber sequence
\[
E\mathcal{F}_{H+} \to S^0 \to \widetilde{E}\mathcal{F}_{H},
\]
then smashing with X and taking $H$-fixed points yields a cofiber sequence, the isotropy separation sequence
\[
(X \sma E\mathcal{F}_{H+})^H \to X^H \to (X\sma \widetilde{E}\mathcal{F}_{H})^H = X^{gH}.
\]
Thus there is always a map $X^H \to X^{gH}$ from the $H$-fixed points of $X$ to the $H$-geometric fixed points of $X$. If $X$ is a structured ring spectrum, this is a map of structured ring spectra. 

We are now ready to recall the definition of cyclotomic spectra. Let $C_n \subset S^1$ denote the cyclic subgroup of order $n$, and let $\rho_n: S^1 \rightarrow S^1/C_n$ denote the isomorphism given by the $n$\textsuperscript{th} root. Using this isomorphism, an $S^1/C_n$-spectrum $E$ indexed on $\mathcal{U}^{C_n}$ determines an $S^1$-spectrum indexed on $\mathcal{U}$. We write this spectrum as $\rho_{n}^*E$. 

\begin{defn}\label{CycloSpec}
A cyclotomic spectrum is a genuine $S^1$-equivariant spectrum $X$ together with equivalences of $S^1$-spectra
\[
r_n: \rho_n^* (X^{gC_n}) \xrightarrow{\simeq} X
\]
for all $C_n \subset S^1$, such that for any $m, n >0$ the following diagram commutes
\[
\xymatrix{ \rho_n^* \Big(\big({\rho_m^*(X^{gC_m}})\big)^{gC_n}\Big) \ar@{=}[r] \ar[d]^-{\rho_n^*\big((r_m)^{gC_n}\big)} & \rho_{nm}^*(X^{gC_{nm}}) \ar[d]^-{r_{mn}} \\ \rho_n^*(X^{gC_n}) \ar[r]^-{r_n} & X }
\]
\end{defn}

We will also need the notion of a cyclotomic map between cyclotomic spectra.

\begin{defn}
A cyclotomic map $f : X \to Y$ between cyclotomic spectra is a map of genuine $S^1$-equivariant spectra such that the diagram
\[
\xymatrix{ \rho_n^* (X^{gC_n}) \ar[r]^-{r_n} \ar[d] & X \ar[d] \\ \rho_n^* (Y^{gC_n}) \ar[r]^-{r_n} & Y }
\]
commutes for all $n$.
\end{defn}

It is clear that cyclotomic spectra together with cyclotomic maps form a category.  A cyclotomic spectrum can be built by taking the suspension spectrum of  a cyclotomic space, as in Example \ref{loops} below, but cyclotomic spectra also arise in other ways.

\begin{example}\label{THH}
For a ring $A$, the topological Hochschild homology $S^1$-spectrum $\THH(A)$ is a cyclotomic spectrum \cite{HeMa97}.
\end{example}

The importance of cyclotomic spectra stems from the fact that a cyclotomic spectrum $X$ comes with the structure required to define the topological cyclic homology of $X$, $\TC(X)$, following \cite{ BHM}. This is done as follows. For any $S^1$-spectrum $X$, there are maps $F_n : X^{C_{mn}} \to X^{C_m}$ for all $n,m\geq 1$, given by inclusion of fixed points. These maps are called Frobenius maps. If $X$ is cyclotomic, there is a second class of maps: $R_n : X^{C_{mn}} \to X^{C_m}$, for all $n,m \geq 1$. These maps, called restriction maps, are given as composites
\[
X^{C_{mn}} = \big(\rho_n^* (X^{C_n})\big)^{C_m} \to \big(\rho_n^* (X^{gC_n})\big)^{C_m} \xrightarrow{r_n^{C_m}} X^{C_m}
\]
where the first map is the map from fixed points to geometric fixed points described earlier in the section. 
The topological cyclic homology of $X$ is then defined as
\[
\TC(X) = \holim_{R,F} X^{C_n}.
\]
As discussed in Example \ref{THH}, for a ring $A$, the $S^1$-spectrum $\THH(A)$ is cyclotomic. Hence we can take the topological cyclic homology of this spectrum. The resulting spectrum $\TC\big(\THH(A)\big)$ is usually denoted $\TC(A)$, and we will use this convention as well. This is the topological cyclic homology of the ring $A$.  
For our application to algebraic $K$-theory, three sources of cyclotomic spectra will be particularly important:
\begin{enumerate}
\item Topological Hochschild homology of rings
\item Homotopy fibers of cyclotomic maps
\item Suspension spectra of cyclotomic spaces
\end{enumerate}
The first source of cyclotomic spectra, topological Hochschild homology, has been used extensively in algebraic $K$-theory computations. The consideration of homotopy fibers of cyclotomic maps as cyclotomic spectra is used implicitly in work of Blumberg and Mandell \cite{BlMa12}. The third approach can be thought of as building cyclotomic spectra by hand. Below, we describe in detail the latter two approaches to generating cyclotomic spectra. 

\subsection{Cyclotomic spectra arising as homotopy fibers}

For a cyclotomic map $f : X \to Y$, it follows immediately that the induced maps $X^{C_n} \to Y^{C_n}$ commute with $R$ and $F$. Hence we get an induced map $\TC(X) \to \TC(Y).$

\begin{proposition}\label{homotopyfiber}
Suppose $f : X \to Y$ is a cyclotomic map. Then
\begin{enumerate}
\item The homotopy fiber $\hofib(f)$ is cyclotomic.
\item There is a natural equivalence $ \TC\big(\hofib(f)\big)\xrightarrow{\simeq} \hofib\big(\TC(X) \rightarrow \TC(Y)\big)$

\end{enumerate}
\end{proposition}

\begin{proof}
The homotopy fiber is taken in the category of $S^1$-spectra, so it is an $S^1$-spectrum. Taking fixed points or geometric fixed points preserves homotopy (co)fiber sequences, so we get a diagram
\[
\xymatrix{ \rho_n^*\big( \hofib(f)^{gC_n}\big) \ar[r] \ar@{-->}[d] & \rho_n^* (X^{gC_n}) \ar[r] \ar[d]^{\simeq} & \rho_n^*( Y^{gC_n} )\ar[d]^{\simeq} \\ \hofib(f) \ar[r] & X \ar[r] & Y}
\]
where the rows are homotopy (co)fiber sequences. It follows that we have an equivalence $\rho_n^{*} \big(\hofib(f)^{gC_n}\big) \to \hofib(f)$ of $S^1$-spectra. For the second claim, note that both $\TC$ and $\hofib$ are homotopy limits, so they commute. Hence we have a chain of equivalences
\begin{multline*}
\TC(\hofib(f)) = \holim_{R,F} \hofib(X \to Y)^{C_n} \xrightarrow{\simeq} \holim_{R,F} \hofib(X^{C_n} \to Y^{C_n}) \\
\xrightarrow{\simeq} \hofib(\holim_{R,F} X^{C_n} \to \holim_{R,F} Y^{C_n} ) = \hofib(\TC(X) \to \TC(Y)).\qedhere
\end{multline*}
\end{proof}

\begin{example}
Suppose we have a map $f: A \to B$ of rings. Then the induced map $\THH(A) \to \THH(B)$ is cyclotomic. The homotopy fiber is sometimes denoted $\THH(f)$. If our map is the quotient $A \to A/J$, then the homotopy fiber is usually denoted $\THH(A, J)$. It follows that the two possible definitions of $\TC(f)$, as $\TC(\THH(f))$ or as $\hofib\big(\TC(A) \to \TC(B)\big)$, agree.

More generally, if $J_1,\ldots,J_n$ are ideals in $A$ we define the multi-relative cyclotomic spectrum $\THH(A, J_1,\ldots,J_n)$ as the homotopy pullback, or iterated homotopy fiber, of the $n$-cube which in position $I$ for $I \subset \{1,2,\ldots,n\}$ is $\THH(A/J_I)$, where $J_I $ is the sum of all the $J_i $ with $ i \in I$. It follows that the two possible definitions of $\TC(A,J_1,\ldots,J_n)$ agree. When the ideals are clear from the context we will denote these multi-relative spectra by $\widetilde{\THH}(A)$ and $\widetilde{\TC}(A)$
\end{example}

%
%
\subsection{Suspension spectra of cyclotomic spaces}

The category of cyclotomic spectra is tensored over a corresponding category of cyclotomic spaces.

\begin{defn}
A cyclotomic space $A$ is an $S^1$-equivariant space together with compatible equivalences
\[ 
r_n: \rho_n^* (A^{C_n}) \xrightarrow{\cong} A 
\]
of $S^1$-spaces for all $n$.
\end{defn} 
A map of cyclotomic spaces is defined in the analogous way to that of a cyclotomic spectrum. The $S^1$-equivariant suspension spectrum of a cyclotomic space is a cyclotomic spectrum. Indeed, $\Sigma^\infty_{S^1}(-_+)$ is a functor from unbased cyclotomic spaces to cyclotomic spectra, and $\Sigma^\infty_{S^1}(-)$ is a functor from based cyclotomic spaces to cyclotomic spectra. Below we give several important examples of cyclotomic spaces. A number of these will naturally arise in our algebraic $K$-theory computations. 

\begin{example}\label{loops}
A free loop space $LX$ is a cyclotomic space, 
where $S^1$ acts by rotation of loops. There is an equivariant homeomorphism
\[
r^{-1}\colon LX \rightarrow \rho_n^* \big((LX)^{C_n}\big)
\]
given by $r(\lambda)(z) = \lambda(z^n)$. The $S^1$-suspension spectrum of this space, $\Sigma^{\infty}_{S^1} LX_+$, is a cyclotomic spectrum. 
If we have a map of spaces $X \to Y$, then the induced map $LX \to LY$ is a map of cyclotomic spaces, and by taking suspension spectra we get a map $\Sigma^\infty_{S^1} LX_+ \to \Sigma^\infty_{S^1} LY_+$ of cyclotomic spectra.
\end{example}

\begin{example}\label{monoid}
If $\Pi$ is a pointed monoid, the cyclic bar construction $B^{cy}(\Pi)$ is a based cyclotomic space. This is  relevant to our computation because of the splitting:
\[
\THH(k(\Pi)) \simeq \THH(k) \sma B^{cy}(\Pi).
\]
of \cite{HeMa97}.  If we can describe how $B^{cy}(\Pi)$ is built out of cyclotomic spaces then we get a corresponding description of $\THH(k(\Pi)),$ which can then be used to calculate $\TC(k(\Pi))$.
\end{example}

\begin{example}\label{family}
Suppose $\{A(s)\}_{s\geq 1}$ is a family of $S^1$-equivariant based spaces with compatible homeomorphisms 
\[
\rho_n^* (A(s)^{C_n}) \approx \left\{ \begin{array}{cc} A(s/n) & \textup{if  } n|s \\ * & \textup{if   } n\nmid s \\ \end{array} \right.  
\]
Then
\[
A = \bigvee_{s \geq 1} A(s)
\]
is a based cyclotomic space.
\end{example}

\begin{example} \label{ex:cyclospaceconcrete}
As a concrete example of a family as described in Example \ref{family}, suppose we have real $S^1$-representations $\lambda(s)$ for $s \geq 1$ with the property that $\rho_n^* (\lambda(s)^{C_n}) \cong \lambda(s/n)$ whenever $n \mid s$. Then $\{S^{\lambda(s)}\}_{s \geq 1}$ is such a family. Further, if we let $A(s) = S^1/C_{s+} \sma S^{\lambda(s)}$, where $S^1$ acts on the first coordinate by left-multiplication and on the second by the obvious action (fixing the point at infinity) then $\{A(s)\}_{s\geq 1}$ is also such a family. Therefore 
\[
A = \bigvee_{s \geq 1} S^1/C_{s+} \sma S^{\lambda(s)}
\]
is a based cyclotomic space. 
\end{example}

\section{Algebraic $K$-theory Calculations} \label{s:KTheory}

For a ring $A$, there is a map relating the algebraic $K$-theory of $A$ and the topological cyclic homology of $A$:
\[
trc: K(A) \rightarrow \TC(A).
\]
This map, called the cyclotomic trace, is due to B\"okstedt, Hsiang, and Madsen \cite{BHM}. This map is often close to an equivalence \cite{Mc97, HeMa97, GeHe06}. Indeed, in the case of $A=k[x_1,\ldots,x_n]/(x_1^{a_n},\ldots,x_n^{a_n})$, the cyclotomic trace map on the multi-relative $K$-theory group $\widetilde{K}(A)$:
\[
trc: \widetilde{K}(A) \rightarrow \widetilde{\TC}(A)
\]
is an equivalence. Here the multi-relative $K$-theory spectrum $\widetilde{K}(A)$ is defined as a homotopy pullback in the same way as the multi-relative spectrum $\widetilde{\TC}(A)$. It follows from \cite{Dun97} that this map is an equivalence after $p$-completion. The proof that this can be extended to an integral statement can be found in  \cite{DuGoMc}.

To compute the multi-relative $K$-theory, we will compute the multi-relative topological cyclic homology. As noted above, $\TC(A) = \TC(\THH(A))$, so the first step is understanding the topological Hochschild homology of $A$. For any pointed monoid algebra $k(\Pi)$ there is  splitting \cite{HeMa97}: 
\[
\THH(k(\Pi)) \simeq \THH(k) \wedge B^{cy}(\Pi).
\]
Note that in the single variable case, $k[x]/(x^a)$ is a pointed monoid algebra $k(\Pi_a)$ where $\Pi_a$ is the monoid $\Pi_a = \{0, 1, x, \ldots, x^{a-1} \}, x^a=0$, so we have the splitting
\[
\THH(k[x]/(x^a)) \simeq \THH(k) \wedge B^{cy}(\Pi_a).
\]

To proceed with the $K$-theory computation, one needs to compute the fixed points of topological Hochschild homology.  The $S^1$-equivariant homotopy type of  $B^{cy}(\Pi)$ for a pointed monoid $\Pi$ has only been described in a small number of cases:   \cite{He07} for the pointed monoids  $\{0,1,z,z^2,\ldots\}$ and $\{0,1,x,x^2,\ldots,y,y^2,\ldots\}$ (where in the latter case, $xy=yx=0$), and \cite{HeMa97b} for the  pointed monoid $\Pi_a$.   Let $\widetilde{B}^{cy}(\Pi_a)$ denote the part of the cyclic bar construction involving at least one positive power of $x$ in one of the coordinates. Then
\[
\THH(k) \wedge \widetilde{B}^{cy}(\Pi_a)\simeq \THH(k[x]/(x^a),(x))=\hofib\big(  \THH(k[x]/(x^a)) \to\THH(k)\big)
\] 
is the relative topological Hochschild homology, which we will use to calculate the relative algebraic $K$-theory.   

Hesselholt and Madsen wrote $\widetilde{B}^{cy}(\Pi_a)$ as the homotopy cofiber of a map of $S^1$-spaces:
\begin{equation}
\label{eq:mapofcyclospaces} \bigvee_{s \geq 1} S^1/C_{s+} \sma S^{\lambda(s-1)} \to \bigvee_{s \geq 1} S^1/C_{s+} \sma S^{\lambda(\lfloor \frac{s-1}{a} \rfloor) }. 
\end{equation}
Here $\lambda(s)$ denotes the $S^1$-representation $\bC(1) \oplus \ldots \oplus \bC(s)$, where $\bC(i)$ is $\bC$ with $S^1$ acting by $z \cdot w = z^iw$. The wedge summand indexed by $s$ maps by a multiplication by $a$ map to the wedge summand indexed by $as$.  This can be found as Theorem B in \cite{HeMa97b}, or Theorem 5 in \cite{He05}.

\begin{lemma}\label{le:mapofcyclospaces}
The map in Equation \ref{eq:mapofcyclospaces} is a map of cyclotomic spaces.
\end{lemma}

\begin{proof}
Following Example \ref{ex:cyclospaceconcrete}, both of the spaces in Equation \ref{eq:mapofcyclospaces} are cyclotomic spaces. 
If $r|s$, we have the following commutative diagram:
\[ 
\xymatrix{ \rho_r^*\Big(\big( S^1/C_{s+} \sma S^{\lambda(s-1)} \big)^{C_r}\Big) \ar[d]^{\approx} \ar[r]^-a & \rho_r^*\Big(\big( S^1/C_{as+} \sma S^{\lambda(s-1)} \big)^{C_r} \Big)\ar[d]^{\approx} \\ S^1/C_{s/r+} \sma S^{\lambda(s/r-1)} \ar[r]^-a & S^1/C_{as/r+} \sma S^{\lambda(s/r-1)} .} \]
So the map in Equation  \ref{eq:mapofcyclospaces} preserves the cyclotomic structure.
\end{proof}

To proceed with the calculation, we establish the following notation: Let
\[
X_\varnothing = \bigvee_{s \geq 1} S^1/C_{s+} \sma S^{\lambda(s-1)}
\]
and let
\[
X_{\{i\}} = \bigvee_{s \geq 1} S^1/C_{s+} \sma S^{\lambda(\lfloor \frac{s-1}{a_{i}} \rfloor)},
\]
so that we have  a homotopy cofiber sequence
\[
\xymatrix{X_\varnothing \ar[r] & X_{\{i\}} \ar[r] & \widetilde{B}^{cy}(\Pi_{a_{i}})   }.
\]

We can now generalize this approach to studying the multi-relative topological cyclic homology of truncated polynomials in multiple variables. The ring $A$ is a pointed monoid algebra: 
\[
k[x_1,\ldots,x_n]/(x_1^{a_n},\ldots,x_n^{a_n})\cong k(\Pi_{a_1} \sma \ldots \sma \Pi_{a_n})
\] 
where $\Pi_i$ is defined as above. Thus we can write
\[ 
\THH(A) \simeq \THH(k) \sma B^{cy}(\Pi_{a_1} \sma \ldots \sma \Pi_{a_n}),
\] 
and the multi-relative THH-spectrum, $\widetilde{\THH}(A)$ then splits as
\[ 
\widetilde{\THH}(A) \simeq \THH(k) \sma \widetilde{B}^{cy}(\Pi_{a_1} \sma \ldots \sma \Pi_{a_n}).
\]
Here $\widetilde{B}^{cy}(\Pi_{a_1} \sma \ldots \sma \Pi_{a_n})$ denotes the piece of the cyclic bar construction where the total exponent of all of the  $x_i$ combined is required to be positive.  We can then use the cofiber description of $\widetilde{B}^{cy}(\Pi_{a_{i}})$ to get a description of the equivariant homotopy type of this multi-relative one.

Given $I \subset \{1,\ldots,n\}$, let
\[
X_{I}=\bigwedge_{i=1}^{n}X_{I\cap \{i\}}.
\]

Then smashing together all of the cofiber sequences used to describe $B^{cy}(\Pi_{a_{i}})$ gives us an $n$-cube $\cX=\{X_{I}\}_{I\subset \{1,\dots,n\}}$, the maps of which are the obvious maps from the Hesselholt-Madsen $1$-variable case. The following is immediate from the defining cofiber sequences.

\begin{proposition}
The space $\widetilde{B}^{cy}(\Pi_{a_1} \sma \ldots \sma \Pi_{a_n})$ is the total homotopy cofiber of the $n$-cube $\cX$.
\end{proposition}

For computations, it is helpful to provide another description of $X_{I}$. Let $\chi_{I}$ be the characteristic function of $I$, evaluating to $1$ if the argument is in $I$ and $0$ otherwise. Then let
\begin{equation} \label{eq:lambdaI} 
\lambda_I(\vect{s}) = \bigoplus_{i=1}^n \lambda\left(\left\lfloor \frac{s_i-1}{a_i^{\chi_{I}(i)}} \right\rfloor\right)
\end{equation}
with $\lambda(s)=\bC(1) \oplus \ldots \oplus \bC(s)$ as before. Then distributing the smash products over the wedges gives an equivalence of cyclotomic spectra:
\begin{equation} \label{eq:XI} 
 X_I \simeq \bigvee_{s_1,\ldots,s_n \geq 1} S^1/C_{s_1+} \sma \ldots \sma S^1/C_{s_n+}\sma S^{\lambda_I(\vect{s})}.
\end{equation}
Smashing the $n$-cube with $T = \THH(k)$ we find that $\widetilde{\THH}(A)$ is the iterated homotopy cofiber of the cube $T \sma \cX=\{T \sma X_I\}_{I \subset \{1,\ldots,n\}}$. Checking the condition of Definition \ref{CycloSpec} directly, the smash product of a cyclotomic spectrum with a cyclotomic space is again cyclotomic. It follows from Example \ref{ex:cyclospaceconcrete} and Lemma \ref{le:mapofcyclospaces} that the spectra $T \sma X_I$ and maps in the cube $T \sma \cX$ are cyclotomic. Proposition~\ref{homotopyfiber} then gives the following proposition.
\begin{proposition}\label{ref:IdentificationofTC}
The multi-relative topological cyclic homology is given as an iterated homotopy cofiber
\[
\widetilde{\TC}(A) = \hocofib\big(\{\TC(T \sma X_I)\}_{I \subset \{1,\ldots,n\}}\big).
\]
\end{proposition}
Applying homotopy groups then gives us a Mayer-Vietoris spectral sequence, and this is the spectral sequence in Theorem~\ref{t:main2}.
\begin{cor}\label{cor:SS}
There is a spectral sequence with $E_{1}$-page given by 
\[
E_{1}^{s,t}=\bigoplus_{|I|=s}\TC_{t}(T\sma X_{I})
\]
and converging to $\TC_{t-s}(A)$. This spectral sequence collapses at $E_{n}$.
\end{cor}

Each $\TC(T \sma X_I)$ for  $I \subset \{1, \ldots n\}$ will be computed in two stages: taking the homotopy limit of $\THH(T \sma X_I)^{C_n}$ over the Frobenius maps to get the spectrum $\TF(T \sma X_I)$, and then taking  the homotopy equalizer of the identity and the restriction maps on $\TF(T \sma X_I)$.  To find $\THH(T \sma X_I)^{C_n}$, it will be convenient to write the structure of $X_I$ as an $S^1$-space in a different way.

Let $S^1(s)$ denote $S^1$ with an accelerated $S^1$-action. This is $S^1$-equivariantly homeomorphic to $S^1/C_s$. We take a matrix $A=[a_{ij}] \in M_n(\bZ)$ to represent the map $(S^1)^n \to (S^1)^n$ given by
\[ (z_1,\ldots,z_n) = (z_1^{a_{11}} \cdots z_n^{a_{1n}},\ldots,z_1^{a_{n1}} \cdots z_n^{a_{nn}}).\]
If $A$ is invertible then this map is a homeomorphism, with inverse represented by $A^{-1}$.

\begin{lemma}\label{lem:Euclid}
Given integers $s_{1}$, $s_{2}$, and $a$, let $g=\gcd(s_{1},s_{2})$. Write $a=de$ with $eg=\gcd(s_{1},as_{2})$. The Euclidean algorithm shows that there are matrices
\[
A, B, B', C\in GL_{2}(\mathbb Z)
\]
such that
\[
A\begin{bmatrix} s_{1} \\ s_{2} \end{bmatrix} =  \begin{bmatrix} g \\ 0 \end{bmatrix},\quad
B\begin{bmatrix} s_{1} \\ es_{2} \end{bmatrix} = B' \begin{bmatrix} s_{1} \\ es_{2} \end{bmatrix} = \begin{bmatrix} eg \\ 0 \end{bmatrix},\quad
C\begin{bmatrix} s_{1} \\ as_{2} \end{bmatrix} =  \begin{bmatrix} eg \\ 0 \end{bmatrix},\quad
\]

\[
\begin{bmatrix} e & 0 \\ 0 & 1 \end{bmatrix} A= B \begin{bmatrix} 1 & 0 \\ 0 & e \end{bmatrix},
\]
and
\[
\begin{bmatrix} 1 & 0 \\ 0 & d \end{bmatrix} B' = C\begin{bmatrix} 1 & 0 \\ 0 & d \end{bmatrix}.
\]

Hence the diagram of $S^{1}$-equivariant maps
\[
\xymatrix{
{S^{1}(s_{1})\times S^{1}(s_{2})} \ar[d]_{\left[\begin{smallmatrix} 1 & 0 \\ 0 & a \end{smallmatrix}\right]} \ar[r]^{A}_{\cong} & {S^{1}(g)\times S^{1}(0)}\ar[r]^{\left[\begin{smallmatrix} e & 0 \\ 0 & 1 \end{smallmatrix}\right]} & {S^{1}(eg)\times S^{1}(0)} \ar[d]^{B'B^{-1}}_{\cong} \\
{S^{1}(s_{1})\times S^{1}(as_{2})} \ar[r]_{C}^{\cong} & {S^{1}(eg)\times S^{1}(0)} & {S^{1}(eg)\times S^{1}(0)}\ar[l]^{\left[\begin{smallmatrix} 1 & 0 \\ 0 & d \end{smallmatrix}\right]}}
\]
commutes.
\end{lemma}

The proof is straightforward, with the algorithm for constructing $B$ based on $A$ and the algorithm for constructing $B'$ based on $C$.

The map $A$ in Lemma~\ref{lem:Euclid} provides $S^{1}$-equivariant homeomorphisms
\[
S^{1}(s_{1})\times S^{1}(s_{2})\cong S^{1}(\gcd(s_{1},s_{2}))\times S^{1}(0),
\]
and Lemma~\ref{lem:Euclid} is saying additionally that we can compatibly choose these homeomorphisms to understand the effect of the $a$\textsuperscript{th} power map on one of the factors.

\begin{remark}\label{rem:HighDimTorus}
By fixing a homeomorphism
\[
S^1(s_1) \times \ldots \times S^1(s_{n-1}) \to S^1(\gcd(s_1,\ldots,s_{n-1})) \times \bT^{n-2}
\]
we can reduce the case $n > 2$ to the case $n = 2$. In particular the multiplication by $d$ map in Lemma \ref{lem:Euclid} is concentrated in a single torus factor.
\end{remark}

This implies that we can write the $X_I$ defined in Equation \ref{eq:XI} as
\[
X_I \cong \bigvee_{\vect{s}\in\bN^n} \bT^{n-1}_+ \sma S^1/C_{\gcd(\vect{s})+} \sma S^{\lambda_I(\vect{s})}.
\]
Now we define
\begin{equation} \label{eq:XhatI} 
\widehat{X}_I = \bigvee_{\vect{s}\in\bN^n} S^1/C_{\gcd(\vect{s})_+} \sma S^{\lambda_I(\vect{s})},
\end{equation}
i.e., $X_I$ without the $(n-1)$-torus.

To understand $\TC(T \sma \widehat{X}_I)$ we use the following result (see Equations \ref{eq:XhatI} and \ref{eq:lambdaI} for the notation):

\begin{lemma}\label{l:TF}
Up to profinite completion we have an equivalence
\[ 
\TF(T \sma \widehat{X}_I) \simeq \bigvee_{\vect{s}\in\bN^n} \Sigma [T \sma S^{\lambda_I(\vect{s})}]^{C_{\gcd(\vect{s})}}.
\]
The restriction maps $R_d$ mapping the wedge summand  corresponding to $d\cdot \vect{s}$ to that of $\vect{s}$ is the smash product of restriction map $R_d$ of $T$ and of the homeomorphism 
\[
r_{d}\colon (S^{\lambda_I(d\vect{s})})^{C_{d}}\to S^{\lambda_I(\vect{s})}.
\]
\end{lemma}

\begin{proof}
%
%
This is by  \cite[Lemma 8.2]{HeMa97}, which says that the inclusion of the $S^1$-fixedpoints of an $S^1$-spectrum into the homotopy inverse limit with respect to inclusions of the $C_n$ fixedpoints over all $n\in\bN$ becomes an equivalence after profinite completion.  In our case $S^1$ acts on each wedge summand corresponding to $\vect{s}\in\bN^n$ separately, so
$\TF(T \sma \widehat{X}_I )$ splits as a wedge over $\vect{s}$.  For each 
$\vect{s}$, the Wirthm\"ulcer isomorphism gives an equivalence
\[
\Sigma[F(S^1/{C_{\gcd(\vect{s})}}_+, T \sma S^{\lambda_I(\vect{s})})]\simeq
S^1/C_{\gcd(\vect{s})_+} \sma T \sma S^{\lambda_I(\vect{s})}
\]
(with trivial $S^1$-action on the suspension of the function space), meaning that the $S^1$-fixedpoints of the spectrum on the right are the same as those of the spectrum on the left. Since the function spectrum is conduced, the fixed points are exactly $\Sigma [T \sma S^{\lambda_I(\vect{s})}]^{C_{\gcd(\vect{s})}}$. 

The restriction maps $R_d$ are induced by those of the original cyclotomic spectrum $T \sma \widehat{X}_I$.
\end{proof}

It follows that up to profinite completion,
\[ \TC(T \sma \widehat{X}_I) = \bigvee_{\gcd(\vect{s})=1} \holim_R \Sigma [T \sma S^{\lambda_I(d\vect{s})}]^{C_d}.\]

\noindent Observe that the indexing set for this wedge is the same as the indexing set that identifies $\bN^{n}$ as a disjoint union of copies of $\bN$. This plays an essential role in our computation. 

Let  $\TF(T\sma \widehat{X}_{I};\vect{s})$ (for $\vect{s}\in\bN^n$)  and $\TC(T\sma \widehat{X}_{I};\vect{s})$ (when $\gcd(\vect{s})=1$) denote the summands corresponding to $\vect{s}$ in the respective wedge sums. 

\section{Calculations for perfect fields} \label{s:perfectfields}
For any ring $R$, there is a close connection between $\THH(R)$ and Witt vectors. In fact, we have an isomorphism
\[
\pi_0 \THH(R)^{C_m} \cong \bW_{\langle m \rangle}(R), 
\]
where the Witt vectors defined by the truncation set $\langle m \rangle$ from Equation \ref{eq:anglem} \cite[Addendum 3.3]{HeMa97}.

Let $k$ be a perfect field of characteristic $p > 0$. Recall that in this case, by \cite{HeMa97},
\[
\THH_*(k) \cong k[\mu_0] 
\]
with $|\mu_0|=2$, and 
\[ 
\pi_*( \THH(k)^{C_{p^{m-1}}}) = \TR^m_*(k; p) \cong \bW_m(k; p)[\mu_0].
\]
It follows that
\[
\pi_*(\THH(k)^{C_m}) \cong \bW_{\langle m \rangle}(k)[\mu_0] 
\]
for all $m$.

From this, we can recover the big Witt vectors by taking the limit
\[
\lim_{R, m \leq n} \pi_*( \THH(k)^{C_m}) \cong \bW_n(k)[\mu_0].
\]

We can now prove the main computational result (recalling Definitions  \ref{eq:SqI} and \ref{eq:XhatI}). Suggestively mirroring the notation used for the wedge summands above, let 
\[
S_{q}(I;\vect{s})=S_{q}(I)\cap \langle\vect{s}\rangle.
\]

\begin{thm} \label{t:TFcalc}
Let $k$ be a perfect field of characteristic $p >0$. Then for any $I \subset \{1,\ldots,n \}$ and $\vect{s}\in\bN^n$, $\TF_*(T \sma \widehat{X}_I; \vect{s})$ is concentrated in odd degrees, and we have an isomorphism
\[
\TF_{2q-1}(T \sma \widehat{X}_I; \vect{s}) \cong \bW_{S_{q}(I;\vect{s})} (k).
\]
Under this identification, the restriction maps 
\[
(R_d)_*: \TF_{2q-1}(T \sma \widehat{X}_I; d\cdot \vect{s}) \to \TF_{2q-1}(T \sma \widehat{X}_I; \vect{s})
\] 
are the Witt vector restriction maps of Equation \ref{eq:restrict}.
\end{thm}

\begin{proof}

By Lemma \ref{l:TF}, up to profinite completion we have \[ \TF(T \sma \widehat{X}_I) = \bigvee_{\vect{s}\in\bN^n} \Sigma [T \sma S^{\lambda_I(\vect{s})}]^{C_{\gcd(\vect{s})}}.\]
It therefore suffices to compute
\[
\pi_{2q-2}\big((T \sma S^{\lambda_I(\vect{s})})^{C_{\gcd(\vect{s})}}\big).
\]
We consider the more general situation where $\lambda$ is any complex $S^1$-representation, and we will compute 
\[
\pi_n\big((T \sma S^{\lambda})^{C_{p^rd}}\big),
\]
for $(p,d)=1$. Since $k$ is a $\mathbb{Z}_{(p)}$-algebra, there is a splitting
\[
(T \sma S^{\lambda})^{C_{p^rd}} \xrightarrow{\simeq} \prod_{e|d} (T \sma S^{\lambda^{C_{d/e}}})^{C_{p^r}},
\]
where the map to the $e$\textsuperscript{th} factor is $R_{d/e} \circ F_e$ \cite[Proof of Proposition 4.2.5]{HeMa97b}.  By \cite[Proposition 9.1]{HeMa97}, the homotopy groups are concentrated in even degrees and
%
%
\[
\pi_{2q}\big((T \sma S^{\lambda^{C_{d/e}}})^{C_{p^r}}\big) \cong \begin{cases} \bW_s(k;p) & \!\! \dim_{\bC}(\lambda^{C_{(d/e)(p^{r-s+1})}}) \! \leq \! q \! < \!  \dim_{\bC}(\lambda^{C_{(d/e)(p^{r-s})}})  \\ \bW_{r+1}(k;p)  & \!\! q \geq \dim_{\bC}(\lambda^{C_{d/e}}) \end{cases}
\]

We recast this last isomorphism in a form more amenable to comparing with the Witt vectors on $\mathbb N^{n}$.
Recall that for any truncation set $S$, there is a splitting
\[
\mathbb{W}_S(k) \cong \prod_{p \nmid e} \bW_{p^{\mathbb{N}_0} \cap S/e}(k),
\]
where the projection onto the $e$\textsuperscript{th} factor is given by $R_{d/e} \circ F_e$. Hence $\mathbb{W}_S(k)$ splits as a product of $p$-typical Witt vectors.
Let $S_{q}(p^{r}d)$ be the truncation set of all divisors $m$ of $p^rd$ such that $\dim_{\bC}(\lambda^{C_{p^rd/m}}) \leq q$. If $s$ is the unique integer such that 
\[
\dim_{\bC}(\lambda^{C_{(d/e)(p^{r-s+1})}}) \leq q <  \dim_{\bC}(\lambda^{C_{(d/e)(p^{r-s})}}),
\] 
and $s=r+1$ if $q \geq \dim_{\bC}(\lambda^{C_{d/e}})$, then
\[
p^{\mathbb{N}_0} \cap S/e = \{ 1, p, p^2, \ldots, p^{s-1} \}.
\]
It follows that
\[
\pi_{2q}\big((T \sma S^{\lambda})^{C_{p^rd}}\big) \cong \mathbb{W}_{S_{q}(p^{r}d)}(k).
\]

We now specialize to the case we are interested in:
\[
\TF_{2q-1}(T \sma \widehat{X}_I; \vect{s}) = \pi_{2q-2}\big([T \sma S^{\lambda_I(\vect{s})}]^{C_{\gcd(\vect{s})}}\big).
\]
From the calculation above, 
\[
\TF_{2q-1}(T \sma \widehat{X}_I; \vect{s}) \cong  \mathbb{W}_{S_{q-1}(\gcd(\vect{s}))}(k).
\]
Associate to each $m |\gcd(\vect{s})$ a vector $\vect{u} \mid \vect{s}$ via 
\[
\vect{u}=\frac{m}{\gcd(\vect{s})} \vect{s}.
\]
Then $\dim_{\bC}(\lambda_{I}(\vect{s})^{C_{\gcd(\vect{s})/m}})$ is
\[ 
\sum_{i=1}^n \left\lfloor \left(\frac {\Big\lfloor \frac{s_i-1}{a_i^{\chi_I(i)}}\Big \rfloor}{\frac{\gcd(\vect{s})}{m}}\right)\right\rfloor
=\sum_{i=1}^n \left\lfloor \frac{s_i-1}{a_i^{\chi_I(i)}}\frac{m}{\gcd(\vect{s})}\right\rfloor
=\sum_{i=1}^n \left\lfloor \frac{u_i-\frac {m}{\gcd(\vect{s})}}{a_i^{\chi_I(i)}}\right\rfloor
=\sum_{i=1}^n \left\lfloor \frac{u_i-1}{a_i^{\chi_I(i)}}\right\rfloor
\]
because $0< \frac {m}{\gcd(\vect{s})}\leq 1$ and $u_i $ and $a_i^{\chi_I(i)}$ are integers.
Thus the truncation set $S$ is isomorphic to the truncation set
\[
S_{q}(I;\vect{s}).
\]
The result follows. 

The claim about the restriction maps follows from the one-dimensional case, since each  restriction map
\[
R_d: \TF(T \sma \widehat{X}_I; d\vect{s}) \to \TF(T \sma \widehat{X}_I; \vect{s})
\]
is happening in the one-dimensional truncation set $\bN\vect{s}$.
\end{proof}

This allows us to prove Theorem~\ref{t:main2}.

\begin{proof}[Proof of Theorem \ref{t:main2}]
By Theorem \ref{t:TFcalc}, if we take the limit 
\[
\holim_{R}\TF_{2q-1}(T \sma \widehat{X}_I; \vect{s})
\] 
over all $\vect{s} \in S_q(I)$, then we get $\bW_{S_q(I)}(k)$. There is an equivalence between the limit over $R$ of this system and the equalizer of $R$ and the identity function  over all $\vect{s}$, so the result follows. The calculation of the order when $k$ is finite follows from the formula
\[
|S_q(I)| = \binom{n+q-1}{n} \prod_{i \in I} a_i.
\]
\end{proof}

We have now computed the homotopy groups at each vertex of our cube. The maps in the $n$-cube are constructed from the $a_i$\textsuperscript{th} power maps on the various factors. In the $1$-variable case studied by Hesselholt and Madsen \cite{HeMa97b} that yielded the ordinary Verschiebung map $V_a$. In this case the situation is a little bit more complicated. On the cube $\TF(\THH(k) \sma \cX)$ we find the following.

\begin{lemma} \label{l:aasjsdfkl}
Let $k$ be any ring. For each $\vect{s}=(s_1,\ldots,s_n)$, let 
\[
\vect{u}=(s_1,\ldots,a_is_i,\ldots,s_n).
\] 
Let $e_i=\gcd(\vect{u})/\gcd(\vect{s})$ and let $\vect{t}=(s_1,\ldots,e_i s_i,\ldots,s_n)$. Up to profinite completion, and using the factorization in Lemma~\ref{lem:Euclid}, we can factor the $a_i$\textsuperscript{th} power map
\[ 
\TF(\THH(k) \sma X_{I-\{i\}};\vect{s}) \to \TF(\THH(k) \sma X_I; \vect{u}) 
\]
as the composite of
\[
\Sigma \big[ \THH(k) \sma S^\lambda \big]^{C_{\gcd(\vect{s})}} \sma \bT^{n-1}_+ \xrightarrow{V_{e_i} \sma 1} \Sigma \big[ \THH(k) \sma S^\lambda \big]^{C_{\gcd(\vect{t})}} \sma \bT^{n-1}_+,
\]
an equivalence on $\Sigma \big[ \THH(k) \sma S^\lambda \big]^{C_{\gcd{(\vect{t})}}} \sma \bT^{n-1}_+$, and
\[
\Sigma \big[ \THH(k) \sma S^\lambda \big]^{C_{\gcd(\vect{t})}} \sma \bT^{n-1}_+ \xrightarrow{1 \sma d_i} \Sigma \big[ \THH(k) \sma S^\lambda \big]^{C_{\gcd(\vect{u})}} \sma \bT^{n-1}_+
\]
Here $d_i$ denotes multiplication by $d_i$ on the last torus factor.
\end{lemma}

\begin{proof}
This follows from Remark \ref{rem:HighDimTorus} by smashing with $THH(k) \sma S^\lambda$ and then applying Lemma \ref{l:TF}.
\end{proof}

\begin{proof}[Proof of Theorem \ref{t:main3}]
Suppose $p$ does not divide any of the $a_i$. Then both the Verschiebung $V_{e_i}$ and a map of degree $d_i$ on the last torus factor as in Lemma \ref{l:aasjsdfkl} above induce isomorphisms. Hence $\widetilde{\TC}(A)$ splits as a wedge over those $\vect{s}$ for which $a_i \nmid s_i$ for all $i$. The result follows. The calculation of the order follows from that in Theorem \ref{t:main2}.
\end{proof}

\section{Calculations for $\bZ$} \label{s:integers}
In the single variable case, the algebraic $K$-theory groups $K_q(\mathbb{Z}[x]/x^n, (x))$ have been studied by the first two authors and Hesselholt \cite{AnGeHe}. They compute these groups completely for $q$ odd and up to extensions for $q$ even. We now consider the $n$-variable case, proving Theorem \ref{mainZ}, stated in the introduction. To prove the theorem we compute the Poincar\'e series of $\TC(T \sma X_I)$ at each vertex of our $n$-cube $\TC(T \sma \cX)$. 

\begin{prop} \label{PoincareTCZ}
For $k=\bZ$ the Poincar\'e series of $\TC(T \sma \widehat{X}_I)$ is given by
\[ \sum_{j \geq 1} t^{2j-1} \binom{n+j-2}{n-1} \prod_{i \in I} a_i. \]
\end{prop}

\begin{proof}

We use the formula from Lemma \ref{l:TF}:  up to profinite completion,
\[
\TF(T \sma \widehat{X}_I) = \bigvee_{\vect{s}\in\bN^n} \Sigma (T \sma S^{\lambda_I(\vect{s})})^{C_{\gcd(\vect{s})}}.
\]
Using the computations of $RO(S^1)$-graded TR-groups of $\mathbb{Z}$ found in \cite{AnGeHe}, we know that
\[
\Sigma(\THH(\bZ) \sma S^{\lambda_I(\vect{s})})^{C_{\gcd(\vect{s})}},
\]
contributes a $\bZ$ in degree $\big(2\dim_\bC(S^{\lambda_I(\vect{s})})^{C_d}+1\big)$ for each $d \mid \gcd(\vect{s})$.
Each such $\bZ$ comes from $\pi_0(\THH(\bZ) )\cong\bZ$ and the sphere
$[S^{\lambda_I(\vect{s})}]^{C_{d}}$.  Note that the restriction maps of the cyclotomic spectrum $\THH(\bZ) $ induce the identity on $\pi_0$, and that the restriction maps $r_m: (S^{\lambda_I(m\vect{s})})^{C_m} \to S^{\lambda_I(\vect{s})}$ exactly send $(S^{\lambda_I(m\vect{s})})^{C_{md}} $ homeomorphically onto $(S^{\lambda_I(\vect{s})})^{C_{d}}$.  So in $ \TC_*(T \sma \widehat{X}_I) $, we will get a $\bZ$ in degree $\big(2\dim_\bC(S^{\lambda_I(\vect{s})})+1\big)$ for every $\vect{s}\in\bN^n$.

Therefore, to compute the free rank of $ \TC_{2j-1}(T \sma \widehat{X}_I) $, we need to count the number of vectors $\vect{s}$ such that $2\dim_\bC(S^{\lambda_I(\vect{s})})+1 = 2j-1.$ Hence we want to count the number of $\vect{s}=(s_1,\ldots, s_n)$ such that $\dim_\bC(S^{\lambda_I(\vect{s})}) = j-1.$  There are $\binom{n+j-2}{n-1}$ vectors $(b_1,\ldots,b_n)$ with $b_i\geq 0$ for all $1\leq i\leq n$ and $b_1+\ldots+b_n=j-1$, and there are
\[
\begin{cases} 1 & \textnormal{if $i \not \in I$} \\ a_i & \textnormal{if $i \in I$} \end{cases}
\]
positive integers $s_i$  with
\[
\dim_\bC \lambda\left(\left\lfloor \frac{s_i-1}{a_i^{\chi_{I}(i)}} \right\rfloor\right) = b_i.
\]
The result follows.
\end{proof}

Theorem \ref{PoincareKZ} follows directly from Proposition \ref{PoincareTCZ} and the computation of Verschiebung maps in \cite{AnGeHe}.

\bibliographystyle{plain}
\bibliography{b}

\end{document}